\documentclass[11pt]{article}

\usepackage{jcd}
\bluehyperref

\usemedgeometry

\usepackage[numbers]{natbib}
\usepackage{algorithm}
\usepackage{algorithmic}
\usepackage{overpic}

\usepackage{xr}
\newcommand{\eqdist}{\stackrel{\textup{dist}}{=}}

\newcommand{\Risk}{\textup{Risk}}
\newcommand{\risk}{\Risk}
\newcommand{\wtheta}{\what{\theta}}
\newcommand{\exopt}{\textup{ExOpt}}

\newcommand{\sure}{\textup{SURE}}
\newcommand{\edf}{\textup{edf}}
\newcommand{\df}{\textup{df}}

\newcommand{\proj}{H}
\providecommand{\subopt}{_\star}
\newcommand{\ropt}{r\subopt}

\newcommand{\projnorm}{h_{\textup{op}}}

\begin{document}

\title{A comment and erratum on ``Excess Optimism: How Biased is the
  Apparent Error of an Estimator Tuned by SURE?''}

\author{Maxime Cauchois, Alnur Ali, and John Duchi}

\maketitle

\begin{abstract}
  We identify and correct an error in the paper ``Excess Optimism: How
  Biased is the Apparent Error of an Estimator Tuned by SURE?'' This
  correction allows new guarantees on the excess degrees of freedom---the
  bias in the error estimate of Stein's unbiased
  risk estimate (SURE) for an estimator tuned by directly
  minimizing the SURE criterion---for
  arbitrary SURE-tuned linear estimators.
  Oracle inequalities follow as a consequence of these results
  for such estimators.
\end{abstract}

\section{Introduction and setting}

In Tibshirani and Rosset's paper~\cite{TibshiraniRo19}, they consider
the Gaussian sequence model, where
\begin{equation}
  \label{eqn:gaussian-sequence}
  Y = \theta_0 + Z,
  ~~~
  Z \sim \normal(0, \sigma^2 I_n)
\end{equation}
for an unknown vector $\theta_0 \in \R^n$.  Their paper develops theoretical
results on the \emph{excess optimism}, or the amount of downward bias in
Stein's Unbiased Risk Estimate (SURE), when using SURE to select an
estimator. In their analysis of subset regression estimators~\cite[Section 4]{TibshiraniRo19}, there is an error in their calculation of this
optimism, which we identify and address in this short note. In addition, we
generalize the conclusions of their results, showing how the excess optimism
bounds we develop (through their inspiration) extend to essentially
arbitrary linear estimators. Their major conclusions and
oracle inequalities for SURE-tuned subset regression estimators thus
remain true and extend beyond projection estimators.


For any estimator $\what{\theta} : \R^n \to \R^n$
of $\theta_0$, we define the
risk
\begin{equation*}
  \Risk\big(\what{\theta}\,\big)
  \defeq \E\left[\ltwos{\what{\theta}(Y) - \theta_0}^2
    \right].
\end{equation*}
We study estimation via \emph{linear smoothers}~\cite{BujaHaTi89}, by which
we mean simply that there exists a collection of matrices $\{
\proj_s \}_{s \in S} \subset \R^{n \times n}$, indexed by $s \in S$, where
$S$ is a finite set. Each of these induces an estimator
$\what{\theta}_s(Y) \defeq \proj_s Y$ of $\theta_0$, with risk
\begin{align*}
  R(s) \defeq
  \Risk(\what{\theta}_s)
  = \norm{(I-\proj_s)\theta_0}_2^2 + \sigma^2 \tr(\proj_s^T \proj_s)
  = \ltwo{(I - \proj_s) \theta_0}^2 + \sigma^2 \lfro{\proj_s}^2.
\end{align*}
We write
\begin{equation*}
  p_s \defeq \tr(\proj_s)
\end{equation*}
for the (effective) degrees of freedom~\cite{Efron12} of $\what{\theta}_s(Y)
= \proj_s Y$, i.e., $\df(\what{\theta}_s) = \frac{1}{\sigma^2} \sum_{i =
  1}^n \cov(\what{Y}_i, Y_i)$ for $\what{Y}_i = [\what{\theta}(Y)]_i =
[\proj_s Y]_i$.  The familiar SURE risk criterion is then
\begin{equation*}
  \sure(s) \defeq \ltwo{Y - \proj_s Y}^2 + 2 \sigma^2 p_s,
\end{equation*}
which satisfies
\begin{equation*}
  \E[\sure(s)] = \ltwo{(I - \proj_s) \theta_0}^2 + \E[\ltwo{(I - \proj_s) Z}^2]
  + 2 \sigma^2 p_s =
  R(s) + n \sigma^2,
\end{equation*}
so that if $Y'$ is an independent draw from the Gaussian sequence
model~\eqref{eqn:gaussian-sequence}, then $\sure$ is an unbiased estimator
of the (prediction) error $\E[\ltwos{\what{\theta}(Y) - Y'}^2]$.

Let $s_0 \in S$ be the oracle choice of
estimator,
\begin{align*}
  s_0 \defeq \argmin_{s \in S} \left\{\Risk(\wtheta_s) \right\},
\end{align*}
and let $\what{s}(y)$ be the estimator minimizing the SURE criterion,
\begin{align*}
  \what{s}(Y) \defeq \argmin_{s \in S} \{ \sure(s) =
  \ltwo{Y - \proj_s Y}^2 + 2\sigma^2 p_s \}.
\end{align*}
\citet{TibshiraniRo19} study the excess risk that using
$\what{s}(Y)$ in place of $s_0$ induces in the estimate
$\what{\theta}_{\what{s}}$ of $\theta_0$, defining the
\emph{excess optimism}
\begin{equation*}
  \exopt(\wtheta_{\what{s}}) \defeq
  \Risk(\what{\theta}_{\what{s}}) + n \sigma^2 - \E[\sure(\what{s}(Y))]
  = \risk(\what{\theta}_{\what{s}}) + n \sigma^2 -
  \E\left[\min_{s \in S} \sure(s)\right],
\end{equation*}
which they conjecture is always nonnegative. Rearranging
this inequality and using that $\E[\min_{s \in S}
  \sure(s)] \le \min_{s \in S} \E[\sure(s)]$ yields the risk upper bound
(cf.~\cite[Thm.~1]{TibshiraniRo19}).
\begin{align}
  \label{eqn:excess-optimism-bound}
  \Risk(\wtheta_{\what{s}}) \le \Risk(\wtheta_{s_0}) + \exopt(\wtheta_{\what{s}}),
\end{align}
so bounding $\exopt$ suffices to control $\Risk(\wtheta_{\what{s}})$. By a
standard calculation~\cite[Sec.~1]{TibshiraniRo19},
\begin{align*}
  \exopt(\wtheta_{\what{s}})
  & = \E[\ltwos{\theta_0 - \proj_{\what{s}(Y)}(Y)}^2] + n \sigma^2
  - \E[\ltwos{\theta_0 + Z - \proj_{\what{s}(Y)}(Y)}^2
    + 2 \sigma^2 p_{\what{s}(Y)}] \\
  & = 2 \E\left[(\proj_{\what{s}(Y)}(Y)  - \theta_0)^T Z]
  - \sigma^2 p_{\what{s}(Y)}\right].
\end{align*}
Note that $Z = Y - \theta_0$ and $\E[Z] = 0$, and define the
estimated $Y_i$ by $\what{Y}_i = [\proj_{\what{s}(Y)}(Y)]_i$.  Then
following~\cite{TibshiraniRo19}, we see that if we define the \emph{excess
  degrees of freedom}
\begin{align}
  \edf(\wtheta_{\what{s}})
  \defeq \E \left[ \frac{1}{\sigma^2} \proj_{\what{s}(Y)}(Y)^T (Y - \theta_0)
    - p_{\what{s}(Y)} \right]
  = \bigg(\sum_{i = 1}^n \frac{1}{\sigma^2} \cov(\what{Y}_i, Y_i)\bigg)
  - \E[p_{\what{s}(Y)}],
  \label{eqn:edf}
\end{align}
the gap between the effective degrees of
freedom and the parameter ``count,'' then
\begin{equation*}
  \exopt(\wtheta_{\what{s}}) = 2 \sigma^2 \edf(\wtheta_{\what{s}}).
\end{equation*}
\citet{TibshiraniRo19} study bounds on this excess degrees of freedom.

\paragraph{The excess degrees of freedom and an error}
By adding and subtracting
$\proj_{\what{s}(Y)}(\theta_0)$ in the definition of the excess
degrees of freedom~\eqref{eqn:edf},
we obtain the equality
\begin{align}
  \edf(\wtheta_{\what{s}})
  = \E\left[ \frac{1}{\sigma^2} Z^T \proj_{\what{s}(Y)}(Z)
    - p_{\what{s}(Y)} \right]
  + \frac{1}{\sigma^2}\E\left[ \proj_{\what{s}(Y)}(\theta_0)^T Z \right].
  \label{eqn:edf-expansion}
\end{align}
This a corrected version of~\cite[Equation~(41)]{TibshiraniRo19}, as the
referenced equation omits the second term. In this note, we provide bounds
on this corrected $\edf$ quantity. The rough challenge is
that, with a naive analysis, the second term in the
expansion~\eqref{eqn:edf-expansion} may scale as
$\sigma\norm{\theta_0}_2\sqrt{\log |S|}$; we can prove that it
is smaller, and the purpose of Theorem~\ref{theorem:edf-truth} in the next
section is to provide a corrected bound on $\edf(\what{\theta}_{\what{s}})$.

\paragraph{Notation} We collect our mostly standard notation here.
For functions $f, g : \mc{A} \to \R$, where $\mc{A}$ is any abstract index
set, we write $f(a) \lesssim g(a)$ to indicate that there is a numerical
constant $C < \infty$, independent of $a \in \mc{A}$, such that
$f(a) \le C g(a)$ for all $a \in \mc{A}$. We let $\log_+ t =
\max\{0, \log t\}$ for $t > 0$. The operator norm of a matrix $A$ is
$\opnorm{A} = \sup_{\ltwo{u} = 1}
\ltwo{A u}$.

\section{A bound on the excess degrees of freedom}

We present a bound on the
excess degrees of freedom $\edf(\what{\theta}_{\what{s}})$ of the
SURE-tuned estimator $\what{\theta}_{\what{s}(Y)} = \proj_{\what{s}(Y)}(Y)$,
where $\what{s}(Y) = \argmin_{s \in S} \sure(s)$.
In the theorem, we recall the oracle estimator
$s_0 = \argmin_{s \in S} \Risk(\what{\theta}_s)$.
\begin{theorem}
  \label{theorem:edf-truth}
  Let $\ropt$ satisfy
  $\frac{1}{\sigma^2} R(s_0)
  = \frac{1}{\sigma^2} \ltwo{(I - \proj_{s_0})\theta_0}^2
  + \lfro{\proj_{s_0}}^2 \le \ropt$. Assume additionally that
  $\opnorm{\proj_s} \le \projnorm$ for all $s \in S$, where
  $\projnorm \ge 1$. Then
  \begin{equation*}
    \edf(\what{\theta}_{\what{s}})
    \lesssim \sqrt{\ropt \log |S|} + \projnorm \log |S| \cdot
      \left(1 + \log_+ \frac{\projnorm^2 \log |S|}{\ropt} \right).
  \end{equation*}
\end{theorem}

We prove the theorem in Section~\ref{sec:proof-edf-truth}, providing a bit
of commentary here by discussing subset regression estimators, as in the
paper~\cite[Sec.~4]{TibshiraniRo19}, then discussing more general smoothers
and estimates~\cite{BujaHaTi89}. As we shall see, the assumption that the
matrices $\proj_s$ have bounded operator norm is little loss of generality~\cite[Sec.~2.5]{BujaHaTi89}.

\subsection{Subset regression and projection estimators}

We first revisit the setting that \citet[Sec.~4]{TibshiraniRo19} consider,
when each matrix $\proj_s$ is an orthogonal projector. As a motiviating
example, in linear regression, we may take $\proj_s = X_s (X_s^T X_s)^{-1}
X_s^T$, where $X_s$ denotes the submatrix of the design $X \in \R^{n \times
  p}$ whose columns $s$ indexes. In this case, as $\proj_s^2 = \proj_s$, we
have $\opnorm{\proj_s} \le 1$ and $\tr(\proj_s) = \lfro{\proj_s}^2 = p_s$.
First consider the case that there is $s\opt$ satisfying $\proj_{s\opt}
\theta_0 = \theta_0$, that is, the true model belongs to the set $S$. Then
because $s_0$ minimizes $\frac{1}{\sigma^2} \ltwo{(I - \proj_{s})
  \theta_0}^2 + p_s$, we have
\begin{equation*}
  \edf(\what{\theta}_{\what{s}})
  \lesssim \sqrt{p_{s\opt} \log |S|} +
  \log|S| \cdot \left(1 + \log_+\frac{\log |S|}{p_{s\opt}}\right).
\end{equation*}
Whenever $p_{s\opt} \gtrsim \log|S|$, we have the simplified bound
\begin{equation*}
  \edf(\what{\theta}_{\what{s}})
  \lesssim \sqrt{p_{s\opt} \log |S|}.
\end{equation*}

An alternative way to look at the result is by replacing
$\ropt$ with $\frac{1}{\sigma^2} R(s_0) = \frac{1}{\sigma^2}
\min_{s \in S} R(s)$. In this case, combining~\cite[Thm.~1]{TibshiraniRo19}
with Theorem~\ref{theorem:edf-truth}, we obtain
\begin{align*}
  \Risk(\what{\theta}_{\what{s}})
  & \le R(s_0)
  + C
  \left(\sqrt{R(s_0) \cdot \sigma^2 \log |S|}
  + \sigma^2 \log |S| \cdot \left(1 + \log_+ \frac{\sigma^2 \log |S|}{R(s_0)}
  \right)\right)
\end{align*}
for a (numerical) constant $C$.
In particular, if $\min_{s \in S} \Risk(\what{\theta}_s) \gtrsim
\sigma^2 \log|S|$, then we obtain that
there exists a numerical constant $C$ such that for all $\eta > 0$,
\begin{align}
  \label{eqn:almost-oracle-inequality}
  \Risk(\what{\theta}_{\what{s}})
  & \le (1 + C \eta) \min_{s \in S} \risk(\what{\theta}_s)
  + C \frac{\sigma^2 \log|S|}{\eta}.
\end{align}

Additional perspective comes by considering some of the bounds
\citet[Sec.~4]{TibshiraniRo19} develop. While we cannot recover
identical results because of the correction
term~\eqref{eqn:edf-expansion} in the excess degrees of freedom, we can
recover analogues. As one example, consider their Theorem 2. It assumes that
$n \to \infty$ in the sequence model~\eqref{eqn:gaussian-sequence}, and
(implicitly letting $S$ and $s_0$ depend on $n$) that
$\risk(\what{\theta}_{s_0})/ \log |S| \to \infty$.  Then
inequality~\eqref{eqn:almost-oracle-inequality} immediately gives the oracle
inequality $\risk(\what{\theta}_{\what{s}}) = (1 + o(1))
\risk(\what{\theta}_{s_0})$, and Theorem~\ref{theorem:edf-truth} moreover
shows that
\begin{equation*}
  \frac{\edf(\what{\theta}_{\what{s}})}{\frac{1}{\sigma^2}
    \risk(\what{\theta}_{s_0})}
  \lesssim \sqrt{\sigma^2 \frac{\log|S|}{\risk(\what{\theta}_{s_0})}}
  \to 0.
\end{equation*}
In brief, \citeauthor{TibshiraniRo19}'s major conclusions on subset
regression estimators---and projection estimators more generally---hold, but
with a few tweaks.

\subsection{Examples of more general smoothers}

In more generality, the matrices $\proj_s \in \R^{n \times n}$ defining
the estimators may be
arbitrary. In common cases, however, they are reasonably
well-behaved. Take as motivation nonparametric regression, where $Y_i =
f(X_i) + \varepsilon_i$, and we let $\theta_i = f(X_i)$ for $i = 1, \ldots,
n$ and $X_i \in \mc{X}$, so that $\Risk(\what{\theta})$ measures the
in-sample risk of an estimator $\what{\theta}$ of $[f(X_i)]_{i = 1}^n$. Here,
standard estimators include kernel regression and
local averaging~\cite{BujaHaTi89, Wainwright19}, both of which we touch on.

First consider kernel ridge regression (KRR). For a reproducing kernel $K :
\mc{X} \times \mc{X} \to \R$, the (PSD) Gram matrix $G$ has entries $G_{ij}
= K(X_i, X_j)$. Then for $\lambda \in \R_+$, we define $\proj_\lambda = (G +
\lambda I)^{-1} G$, which is symmetric and positive semidefinite, and
satisfies $\opnorm{\proj_\lambda} \le 1$. Assume the collection $\Lambda
\subset \R_+$ is finite and define the effective dimension $p_\lambda =
\tr(G(G + \lambda I)^{-1})$, yielding that $\sure(\lambda) = \ltwos{Y -
  \proj_\lambda Y}^2 + 2 \sigma^2 p_\lambda$.  Then
SURE-tuned KRR, with regularization $\what{\lambda} =
\argmin_{\lambda \in \Lambda} \sure(\lambda)$ and
$\what{\theta}_{\what{\lambda}} = \proj_{\what{\lambda}} Y$, satisfies
\begin{equation*}
  \edf(\what{\theta}_{\what{\lambda}})
  \lesssim \sqrt{\ropt \log |\Lambda|}
  + \log |\Lambda| \cdot \left(1 + \log_+ \frac{\log|\Lambda|}{\ropt}\right)
\end{equation*}
where as usual, $\ropt = 
\min_{\lambda \in \Lambda} \{\frac{1}{\sigma^2}\ltwo{(I - \proj_\lambda) \theta_0}^2
+ \tr(G(G + \lambda I)^{-1})\}$.

A second example arises from $k$-nearest-neighbor ($k$-nn) estimators. We
take $S = \{1, \ldots, n\}$ to indicate the number of nearest neighbors to
average, and for $k \in S$ let $\mc{N}_k(i)$ denote the indices of the $k$
nearest neighbors $X_j$ to $X_i$ in $\{X_1, \ldots, X_n\}$, so
$\mc{N}_k^{-1}(i) = \{j \mid i \in \mc{N}_k(j)\}$
are the indices $j$ for which $X_i$ is a neighbor of $X_j$. Then the
matrix $\proj_k \in \R^{n \times n}_+$ satisfies $[\proj_k]_{ij} =
\frac{1}{k} \indic{j \in \mc{N}_k(i)}$, and we claim that $\opnorm{\proj_k} \le
\frac{1}{k} \max_i |\mc{N}_k^{-1}(i)|$.
Indeed,
\begin{equation*}
  [\proj_k^T \proj_k]_{ij}
  = \frac{1}{k^2} \sum_{l = 1}^n \indic{i \in \mc{N}_k(l),
    j \in \mc{N}_k(l)},
\end{equation*}
and as $H_k^T H_k$ is elementwise nonnegative,
the Gershgorin circle theorem guarantees
\begin{align*}
  \opnorm{\proj_k}^2
  \le \max_{i \le n}
  \sum_{j = 1}^n [\proj_k^T \proj_k]_{ij}
  & \le \frac{1}{k^2} \max_{i \le n} \sum_{l = 1}^n \sum_{j = 1}^n
  \indic{i \in \mc{N}_k(l), j \in \mc{N}_k(l)}
  \le \frac{1}{k} \max_{i \le n} \sum_{l = 1}^n \indic{i \in \mc{N}_k(l)}
\end{align*}
as $|\mc{N}_k(l)| \le k$ for each $l$.
Additionally, we have
$\lfro{H_k}^2 = \frac{n k}{k^2} = \frac{n}{k}$.
The normalized risk of the $k$-nn estimator is then $r_k
= \frac{1}{\sigma^2} \ltwo{(I - H_k) \theta_0}^2 + \frac{n}{k}$,
and certainly $r_k \ge \log n$ whenever $k \le \frac{n}{\log n}$.
Under a few restrictions, we can therefore obtain an oracle-type inequality:
assume the points $\{X_1, \ldots, X_n\}$ are regular enough that
$\max_i |\mc{N}_k^{-1}(i)| \lesssim k$ for $k \le \frac{n}{\log n}$.
Then the SURE-tuned $k$-nearest-neighbor
estimator $\what{\theta}_k$
satisfies the bound
\begin{equation*}
  \edf(\what{\theta}_{\what{k}})
  \lesssim \sqrt{\frac{1}{\sigma^2}
    \min_{k \le n / \log n} \risk(\what{\theta}_k) \cdot \log n}
\end{equation*}
on its excess degrees of freedom.
Via inequality~\eqref{eqn:excess-optimism-bound}, this implies the oracle
inequality
\begin{equation*}
  \risk(\what{\theta}_{\what{k}})
  \le \min_{k \le n / \log n} \left(1 + C \eta\right) \risk(\what{\theta}_k) +
  \frac{C \sigma^2 \log n}{\eta}
  ~~~ \mbox{for~all~} \eta > 0.
\end{equation*}


\subsection{Proof of Theorem~\ref{theorem:edf-truth}}
\label{sec:proof-edf-truth}

Our proof strategy is familiar from high-dimensional
statistics~\cite[cf.][Chs.~7 \& 9]{Wainwright19}: we develop a basic
inequality relating the risk of $\what{s}$ to that of $s_0$, then
apply a peeling argument~\cite{vandeGeer00} to bound the probability
that relative error bounds deviate far from their expectations.
Throughout, we let $C$ denote a universal (numerical) constant
whose value may change from line to line.

To prove the theorem, we first recall a
definition and state two auxiliary lemmas.
\begin{definition}
  \label{definition:sub-exp}
  A mean-zero random variable $X$ is $(\tau^2, b)$-sub-exponential
  if $\E[e^{\lambda X}] \le \exp(\frac{\lambda^2 \tau^2}{2})$ for
  $|\lambda| \le \frac{1}{b}$. If $b = 0$, then $X$ is $\tau^2$-sub-Gaussian.
\end{definition}
\begin{lemma}
  \label{lemma:loads-of-maxima}
  Let $X_i$, $i = 1, \ldots, N$, be $\tau^2$-sub-Gaussian and $k \ge
  1$. Then
  \begin{equation*}
    \E\left[\max_{i \le N} |X_i|^k\right]
    \le 2
    \cdot \tau^k
    \max\left\{ (2 \log N)^{k/2},
    k^{k/2} \right\}.
  \end{equation*}
  Let $X_i$, $i = 1, \ldots,  N$, be $(\tau^2, b)$-sub-exponential
  and $k \ge 1$. Then
  \begin{equation*}
    \E\left[\max_{i \le N} |X_i|^k\right]^{1/k}
    \lesssim \max\left\{\sqrt{\tau^2 \log N}, b \log N,
    \sqrt{\tau^2 k}, bk \right\}.
  \end{equation*}
\end{lemma}
\noindent
The second  statement of the lemma generalizes the first without specifying
constants. We also use that
quadratic forms of Gaussian random vectors are sub-exponential.
\begin{lemma}
  \label{lemma:quadratic-forms}
  Let $Z \sim \normal(0, I_{n \times n})$ and $A \in \R^{n \times n}$.
  Then $Z^T A Z$ is $(\tr((A + A^T)^2), 2 \opnorms{A + A^T})$-sub-exponential.
  Additionally, $Z^T A Z$ is $(\lfro{A}^2, 4 \opnorms{A})$-sub-exponential.
\end{lemma}
\noindent
We defer the proofs of Lemmas~\ref{lemma:loads-of-maxima} and
\ref{lemma:quadratic-forms} to Sections~\ref{sec:proof-loads-of-maxima} and
\ref{sec:proof-quadratic-forms}, respectively.

Recall the notation
$R(s) = \ltwo{(\proj_s - I) \theta_0}^2 + \sigma^2 p_s$, and for $s \in S$
define the centered variables
\begin{equation}
  \label{eqn:W-and-Z-defs}
  \begin{split}
    W_s & \defeq \frac{1}{\sigma^2} Z^T(2 \proj_s - \proj_s^T \proj_s) Z
    + \tr(\proj_s^T \proj_s)
    - 2 p_s , \\
    Z_s & \defeq
    \frac{1}{\sigma^2} \theta_0^T (\proj_s - I)^T (I - \proj_s) Z.
  \end{split}
\end{equation}
A quick calculation shows that for every $s \in S$,
\begin{align*}
  \frac{1}{\sigma^2}\sure(s)
  & = \frac{1}{\sigma^2} R(s) + \frac{1}{\sigma^2}
  \ltwo{Z}^2
  - W_s - 2 Z_s.
\end{align*}
As $\what{s}(Y)$ minimizes the SURE criterion,
we therefore have the basic inequality
\begin{equation}
  \label{eqn:basic}
  \frac{1}{\sigma^2}\left(R(\what{s}(Y)) - R(s_0) \right)
  \le W_{\what{s}(Y)} - W_{s_0} + 2 (Z_{\what{s}(Y)} - Z_{s_0}).
\end{equation}

We now provide a peeling argument using
inequality~\eqref{eqn:basic}. For each $l \in \N$ define the shell
\begin{align*}
  \mc{S}_l \defeq \{ s \in S \mid (2^l-1)\sigma^2  \ropt
  \le R(s) - R(s_0)
  \le (2^{l+1}-1)\sigma^2 \ropt \}.
\end{align*}
The key result is the following lemma.
\begin{lemma}
  \label{lemma:s-hat-should-be-good}
  There exists a numerical constant $c > 0$ such that
  \begin{equation*}
    \P(\what{s}(Y) \in \mc{S}_l)
    \le 2 |\mc{S}_l| \exp\left(-c \frac{2^l \ropt}{\projnorm^2}\right).
  \end{equation*}
\end{lemma}
\begin{proof}
  Note that
  if $\what{s}(Y) \in \mc{S}_l$, we have
  \begin{align*}
    \max_{s \in \mc{S}_l} \left( W_s - W_{s_0}\right) + 2 \max_{s \in \mc{S}_l}
    (Z_s - Z_{s_0}) \ge \max_{s \in \mc{S}_l} \left( W_s - W_{s_0} + 2 (Z_s - Z_{s_0})
    \right) \ge
    \ropt(2^l - 1),
  \end{align*}
  and so it must be the case that at least one of
  \begin{equation}
    \label{eqn:W-or-Z-big}
    \max_{s \in \mc{S}_l} W_s - W_{s_0} \ge \half \ropt (2^l - 1)
    ~~~ \mbox{or} ~~~
    \max_{s \in \mc{S}_l} (Z_s - Z_{s_0}) \ge \frac{1}{4} \ropt(2^l - 1)
  \end{equation}
  occurs.  We can thus bound the probability that $\what{s}(Y) \in \mc{S}_l$
  by bounding the probabilities of each of the events~\eqref{eqn:W-or-Z-big};
  to do this, we that $W_s -
  W_{s_0}$ and $Z_s - Z_{s_0}$ concentrate for $s \in \mc{S}_l$.

  As promised, we now show that $W_s - W_{s_0}$ and $Z_s -
  Z_{s_0}$ are sub-exponential and sub-Gaussian, respectively (recall
  Definition~\ref{definition:sub-exp}).  Observe that
  \begin{align}
    \nonumber
    \frac{1}{\sigma^2}(R(s) - R(s_0)) & =
    \frac{1}{\sigma^2} \ltwo{(I - \proj_s) \theta_0}^2
    - \frac{1}{\sigma^2} \ltwo{(I - \proj_{s_0})\theta_0 }^2 + \lfro{\proj_s}^2
    - \lfro{\proj_{s_0}}^2 \\
    & \ge \lfro{\proj_s}^2 - \ropt
    \label{eqn:R-p-bounds}
  \end{align}
  by assumption that $\ropt \ge \frac{1}{\sigma^2} R(s_0) = \frac{1}{\sigma^2}
  \ltwo{(I - \proj_{s_0}) \theta_0}^2 + \lfro{\proj_{s_0}}^2$
  and that $\ltwo{(I - \proj_s)
    \theta_0}^2 \ge 0$.  In particular,
  inequality~\eqref{eqn:R-p-bounds}
  shows that for each $s \in \mc{S}_l$ we have $\lfro{\proj_s}^2 - \ropt \le (2^{l + 1} -
  1) \ropt$ (and we always have $\lfro{\proj_{s_0}}^2 \le \ropt$), so that
  \begin{equation}
    \lfro{\proj_s}^2 \le 2^{l + 1} \ropt ~~~
    \mbox{and} ~~~
    \lfro{\proj_{s_0}}^2 \le \ropt
    ~~~ \mbox{for~} s \in \mc{S}_l.
    \label{eqn:parameter-count-bounds}
  \end{equation}
  For each $s$ we have
  \begin{equation*}
    W_s - W_{s_0} = \frac{1}{\sigma^2} Z^T (2 \proj_s - 2 \proj_{s_0}
    - \proj_s^T \proj_s + \proj_{s_0}^T \proj_{s_0}) Z
    + \lfro{\proj_s}^2 - \lfro{\proj_{s_0}}^2 - 2(p_s - p_{s_0}),
  \end{equation*}
  while
  \begin{equation*}
    \lfro{2\proj_s - 2\proj_{s_0} - \proj_s^T \proj_s + \proj_{s_0}^T \proj_{s_0}}^2
    \le 4 \cdot \left(4\lfro{\proj_s}^2 + 4\lfro{\proj_{s_0}}^2
    + \projnorm^2 \lfro{\proj_s}^2
    + \projnorm^2 \lfro{\proj_{s_0}}^2\right)
  \end{equation*}
  by assumption that $\opnorm{\proj_s} \le \projnorm$ for all $s \in S$.
  Lemma~\ref{lemma:quadratic-forms} and the
  bounds~\eqref{eqn:parameter-count-bounds} on $\lfro{\proj_s}$ thus
  give that $W_s
  - W_{s_0}$ is $(C \cdot 2^l \projnorm^2 \ropt, C \cdot
  \projnorm^2)$-sub-exponential, so that for $t \ge 0$, a Chernoff bound
  implies
  \begin{align*}
    \P\left(\max_{s \in \mc{S}_l} (W_s - W_{s_0}) \ge
    t \right) & \le |\mc{S}_l| \exp\left(C \cdot 2^l \lambda^2 \projnorm^2
    \ropt
    - \lambda t\right),
  \end{align*}
  valid for $0 \le \lambda \le \frac{1}{C \projnorm}$.
  Taking $t = \half (2^l - 1) \ropt$
  and $\lambda = \frac{1}{C' \projnorm^2}$ yields that
  for a numerical constant $c > 0$,
  \begin{equation*}
    \P\left(\max_{s \in \mc{S}_l} (W_s - W_{s_0}) \ge
    \half (2^l  - 1) \ropt\right) \le
    |\mc{S}_l| \exp\left(-c \frac{2^l \ropt}{\projnorm^2}\right).
  \end{equation*}
  We can provide a similar bound on $Z_s - Z_{s_0}$ for $s \in \mc{S}_l$.  It
  is immediate that $Z_s \sim \normal(0, \frac{1}{\sigma^2} \ltwo{(\proj_s -
    I)^T (I - \proj_s) \theta_0}^2)$. Using
  inequality~\eqref{eqn:R-p-bounds} and that $\frac{1}{\sigma^2} \ltwo{(I -
    \proj_{s_0})\theta_0}^2 + \lfro{\proj_{s_0}}^2 \le \ropt$,
  for each $s \in \mc{S}_l$ we have
  \begin{align*}
    \frac{1}{\sigma^2} \norm{(I-\proj_s)\theta_0}_2^2
    & = \frac{1}{\sigma^2}
    (R(s) - R(s_0)) + \frac{1}{\sigma^2} \ltwo{(I - \proj_{s_0}) \theta_0}^2
    + \lfro{\proj_{s_0}}^2 - \lfro{\proj_s}^2  \\
    & \le \ropt (2^{l+1} - 1) + \ropt - p_s
    \le 2^{l + 1} \ropt.
  \end{align*}
  Similarly, $Z_{s_0} \sim \normal(0, \frac{1}{\sigma^2}
  \ltwo{(\proj_{s_0} - I)^T (I - \proj_{s_0}) \theta_0}^2)$ and
  $\frac{1}{\sigma^2} \ltwo{(I - \proj_{s_0})\theta_0}^2 \le \ropt$.
  Using that $\opnorm{I - \proj_s} \le (1 + \projnorm)$,
  for each $s \in \mc{S}_l$ we have
  $Z_s - Z_{s_0} \sim \normal(0, \tau^2(s))$ for some
  $\tau^2(s) \le C \cdot \projnorm^2 2^l \ropt$. This yields the bound
  \begin{align*}
    \P\left( \max_{s \in \mc{S}_l} (Z_s - Z_{s_0}) \ge
    \frac{1}{4} \ropt(2^l - 1) \right)
    & \le |\mc{S}_l| \exp\left( - c \frac{\ropt(2^l - 1)^2}{2^l \projnorm^2} \right)
    \le |\mc{S}_l| \exp \left( - c' \frac{2^l \ropt}{\projnorm^2} \right),
  \end{align*}
  where $c, c' > 0$ are
  numerical constants. Returning to the events~\eqref{eqn:W-or-Z-big},
  we have shown
  \begin{align*}
    \P(\what{s}(Y) \in \mc{S}_l)
    & \le \P\left(\max_{s \in \mc{S}_l} (W_s - W_{s_0})
    \ge \half (2^l - 1) \ropt\right)
    + \P\left(\max_{s \in \mc{S}_l} (Z_s - Z_{s_0})
    \ge \frac{1}{4} \ropt(2^l - 1) \right) \\
    & \le 2 |\mc{S}_l|
    \exp\left( - c\frac{2^l \ropt}{\projnorm^2} \right)
  \end{align*}
  as desired.
\end{proof}

We leverage the probability bound in Lemma~\ref{lemma:s-hat-should-be-good}
to give our final guarantees. We expand
Define the (centered) linear and quadratic terms
\begin{equation*}
  Q_s \defeq \frac{1}{\sigma^2} Z^T \proj_s Z - p_s
  ~~~ \mbox{and} ~~~
  L_s \defeq \frac{1}{\sigma^2} Z^T \proj_s \theta_0,
\end{equation*}
so that
$\edf(\what{\theta}_{\what{s}}) = \E[Q_{\what{s}(Y)}] + \E[L_{\what{s}(Y)}]$.
Expanding this equality, we have
\begin{equation*}
  \edf(\what{\theta}_{\what{s}})
  = \sum_{l = 0}^\infty \E[Q_{\what{s}(Y)} \indic{\what{s}(Y) \in \mc{S}_l}]
  + \E[L_{\what{s}(Y)} \indic{\what{s}(Y) \in \mc{S}_l}].
\end{equation*}
As in the proof of Lemma~\ref{lemma:s-hat-should-be-good}, the
bounds~\eqref{eqn:parameter-count-bounds} that $\lfro{\proj_s}^2 \le 2^{l +
  1} \ropt$ and Lemma~\ref{lemma:quadratic-forms} guarantee that $Q_s$
is $(2^{l+3} \ropt, 4 \projnorm)$-sub-exponential. Thus we have
\begin{align*}
  \E\left[Q_{\what{s}(Y)} \indic{\what{s}(Y) \in \mc{S}_l}\right]
  & \le \E\left[\max_{s \in \mc{S}_l} Q_s \indic{\what{s}(Y) \in \mc{S}_l}\right]
  \stackrel{(i)}{\le} \E\left[\max_{s \in \mc{S}_l} Q_s^2\right]^{1/2}
  \P(\what{s}(Y) \in \mc{S}_l)^{1/2}
  \\
  & \stackrel{(ii)}{\lesssim} \max\left\{\sqrt{2^l \ropt \log |\mc{S}_l|},
  \projnorm \log |\mc{S}_l|\right\}
  \min\left\{1, \sqrt{|\mc{S}_l|}  \exp\left(-c \frac{2^l \ropt}{\projnorm^2}
  \right)\right\},
\end{align*}
where inequality~$(i)$ is Cauchy-Schwarz and inequality~$(ii)$ follows
by combining Lemma~\ref{lemma:loads-of-maxima} (take $k = 2$) and
Lemma~\ref{lemma:s-hat-should-be-good}.
We similarly have that $L_s$ is $C \cdot 2^l \ropt$-sub-Gaussian,
yielding
\begin{equation*}
  \E[L_{\what{s}(Y)} \indic{\what{s}(Y) \in \mc{S}_l}]
  \lesssim \sqrt{2^l \ropt \log |\mc{S}_l|} \min\left\{1,
  \sqrt{|\mc{S}_l|} \exp\left(-c \frac{2^l \ropt}{\projnorm^2}\right)\right\}.
\end{equation*}
Temporarily introduce the shorthand
$r = \frac{\ropt}{\projnorm^2}$.
Substituting these bounds into the $\edf(\what{\theta}_{\what{s}})$ expansion
above and naively bounding $|\mc{S}_l| \le |S|$ yields
\begin{align*}
  \lefteqn{\edf(\what{\theta}_{\what{s}})} \\
  & \lesssim
  \sqrt{\ropt \log |S|} \int_0^\infty
  \sqrt{2^t \min\{1, |S| \exp(-c 2^t r)\}} dt
  + \projnorm \log |S| \int_0^\infty \sqrt{\min\{1, |S| \exp(-c 2^t r)
    \}} dt \\
  & \qquad \qquad ~ + \sqrt{\ropt \log|S|} + \projnorm \log|S| \\
  & \lesssim \sqrt{\log|S|}
  \int_{cr / 2}^\infty u^{-\half}
  \min\{1, \sqrt{|S|} e^{-u}\} du
  + \projnorm \log |S| \int_{c r / 2}^\infty \frac{1}{u}
  \min\{1, \sqrt{|S|} e^{-u}\} du \\
  & \qquad\qquad ~ + \sqrt{\ropt \log|S|} + \projnorm \log|S|,
\end{align*}
where we  made the substitution $u = c 2^{t - 1} r$.
We break each of
the integrals into the
regions $\half c r \le u \le \half\max\{c r, \log|S|\}$
and $u \ge \half \max\{c r, \log|S|\}$. Thus
\begin{align*}
  \sqrt{\log|S|}
  \int_{c r / 2}^\infty u^{-\half}
  \min\{1, \sqrt{|S|} e^{-u}\} du
  & \le \begin{cases}
    \sqrt{2} \log |S| + \sqrt{2} & \mbox{if}~ \log|S| \ge cr \\
    \sqrt{2} & \mbox{if~} c r \ge \log|S|.
  \end{cases}
\end{align*}
where we have used
$\int_b^\infty u^{-1/2} e^{-u} du \le b^{-1/2} e^{-b}$ for $b > 0$ and
that $\frac{\sqrt{|S| \log|S|}}{\sqrt{c r / 2}} \exp(-\half c r)
\le \sqrt{2}$ whenever $c r \ge \log|S|$. For
the second integral, we have
\begin{align*}
  \log|S|
  \int_{c r / 2}^\infty \frac{1}{u} \min\{1, \sqrt{|S|}e^{-u} \} du
  & \le \begin{cases}
    \log|S| \cdot \log\frac{\log |S|}{c r}
    + 2 & \mbox{if~} \log|S| \ge c r \\
    2 & \mbox{if}~ c r \ge \log|S|,
  \end{cases}
\end{align*}
where we have used that $\int_b^\infty \frac{1}{u} e^{-u} du \le
e^{-b} / b$ for any $b \ge 0$. Replacing $r = \ropt / \projnorm^2$,
Theorem~\ref{theorem:edf-truth}
follows.

\subsection{Proof of Lemma~\ref{lemma:loads-of-maxima}}
\label{sec:proof-loads-of-maxima}
For the first statement, without loss of generality by scaling, we may assume
$\sigma^2 = 1$.  Then for any $t_0 \ge 0$, we may write
\begin{align*}
  \E[\max_{i \le N} |X_i|^k]
  & = \int_0^\infty \P\left(\max_{i \le N} |X_i| \ge t^{1/k}\right) dt \\
  & \le t_0 + 2 N \int_{t_0}^{\infty} \exp\left(-\frac{t^{2/k}}{2}\right) dt \\
  & = t_0 + 2^{k/2} N k 
  \int_{t_0^{2/k} / 2}^\infty
  u^{k/2 - 1} e^{-u} du,
\end{align*}
where we have made the substitution $u = t^{2/k} / 2$, or
$2^{k/2} u^{k/2} = t$. Using~\cite[Eq.~(1.5)]{BorweinCh09},
which states that $\int_x^\infty x^{a-1} e^{-x} dx \le
2 x^{a-1} e^{-x}$ for $x > a - 1$,
we obtain that 
\begin{align*}
  \E[\max_{i \le N} |X_i|^k]
  & \le t_0 + 2^{k/2 + 1} N k 
  \left(\half t_0^{2/k}\right)^{k/2 - 1}
  \exp\left(-\frac{t_0^{2/k}}{2} \right)
\end{align*}
whenever $t_0 \ge k^{k/2}$. Take $t_0 = \max\{(2 \log N)^{k/2},
k^{k/2}\}$ to achieve the result.

The second bound is more subtle.  First, we note
that $(X_i/\tau)$ is $(1, b/\tau)$-sub-exponential. We
therefore prove the bound for $(1, b)$-sub-exponential random
variables, rescaling at the end.
Following a similar argument to that above,
note that $\P(|X_i| \ge t) \le
2  \exp(-\min\{\frac{t}{2b}, \frac{t^2}{2}\})$, and so
\begin{align*}
  \E[|X_i|^k]
  & \le \int_0^\infty \P(|X_i| \ge t^{1/k}) dt
  \le 2 \int_0^\infty \exp\left(-\min\left\{\frac{t^{1/k}}{2b},
  \frac{t^{2/k}}{2} \right\}\right) dt \\
  & = 
  2 \int_0^{b^{-k}} \exp\left(-\frac{t^{2/k}}{2}\right) dt
  + 2 \int_{b^{-k}}^\infty \exp\left(-\frac{t^{1/k}}{2b}\right) dt.
\end{align*}
Making the substitution $u = t^{2/k} / 2$, or
$k u^{k/2 - 1} du = dt$ in the first integral, and
$u = t^{1/k}/(2b)$, or
$(2 b)^k k u^{k-1} du = dt$ in the second, we obtain the bounds
\begin{align*}
  \E[|X_i|^k]
  & \le 2 k \int_0^{b^{-2} / 2} u^{k/2 - 1} e^{-u} du
  + 2^{k+1} b^k k \int_{b^{-2}/2}^\infty u^{k-1} e^{-u} du.
\end{align*}
The first integral term always has upper bound
$2k \Gamma(k/2) = 4 \Gamma(k/2 + 1)$,
while the second has bound~\cite[Eq.~(1.5)]{BorweinCh09}
\begin{equation*}
  \int_{b^{-2}/2}^\infty u^{k-1} e^{-u} du
  \le \begin{cases}
    2 (b^{-2} / 2)^{k - 1} e^{-b^{-2} / 2} &
    \mbox{if}~ b^{-2} \ge 4k \\
    \Gamma(k) & \mbox{otherwise}.
  \end{cases}
\end{equation*}
In the former case---when $b$ is small enough that
$b \le \half / \sqrt{k}$---we note that
\begin{equation*}
  k b^k (b^{-2})^{k-1} e^{-b^{-2} / 2}
  = k \exp\left(-\frac{1}{2b^2} + (k - 2) \log \frac{1}{b}
  \right)
  \le k \exp\left(-2k + \frac{k}{2} \log k - \log k\right)
  \le k^{k/2}.
\end{equation*}
Combining the preceding bounds therefore yields
\begin{equation}
  \label{eqn:momen-sub-exponentials}
  \E[|X_i|^k]^{1/k} \lesssim
  \max\left\{\sqrt{k}, b k\right\}.
\end{equation}

We leverage the moment bound~\eqref{eqn:momen-sub-exponentials} to
give the bound on the maxima. Let $p \ge 1$ be arbitrary. Then
\begin{align*}
  \E\left[\max_{i \le N} |X_i|^k\right]^{1/k}
  & \le \E\left[\max_{i \le N} |X_i|^{kp}\right]^{1/kp}
  \le N^{\frac{1}{kp}}
  \max_{i \le N} \E[|X_i|^{kp}]^\frac{1}{kp}
  \lesssim N^{\frac{1}{kp}}
  \max\left\{\sqrt{kp}, bkp \right\}.
\end{align*}
If $\log N \ge k$, take $p = \frac{1}{k} \log N$ to obtain that
$\E[\max_{i \le N} |X_i|^k]^{1/k}
\lesssim \max\{\sqrt{\log N}, b \log N\}$. Otherwise,
take $p = 1$, and note that
$k \ge \log N$.
To obtain the result with appropriate scaling, use the mapping
$b \mapsto b/\tau$ to see that if the $X_i$ are
$(\tau^2, b)$-sub-exponential, then
\begin{equation*}
  \frac{1}{\tau} \E\left[\max_{i \le N} |X_i|^k\right]^{1/k}
  \lesssim \max\left\{\sqrt{\log N}, \frac{b}{\tau} \log N,
  \sqrt{k}, \frac{bk}{\tau} \right\},
\end{equation*}
and multiply through by $\tau$.

\subsection{Proof of Lemma~\ref{lemma:quadratic-forms}}
\label{sec:proof-quadratic-forms}

Note that $Z^T A Z = \half Z^T(A + A^T) Z$; we prove the result
leveraging $B \defeq \half(A + A^T)$.
As $B$ is symmetric, we can write $B = UDU^T$ for a diagonal matrix $D$ and
orthogonal $U$, and as $Z \eqdist U^T Z$ we can further simplify (with no
loss of generality) by assuming $B$ is diagonal with $B = \diag(b_1, \ldots,
b_n)$. Then $Z^T B Z = \sum_{i = 1}^n b_i Z_i^2$. As $\E[e^{\lambda Z_i^2}]
= 1 / \hinge{1 - 2\lambda}^{1/2}$, we have
\begin{equation*}
  \E[\exp(\lambda Z^T B Z)]
  = \exp\left(-\half\sum_{i = 1}^n \log \hinge{1 - 2 \lambda b_i}\right).
\end{equation*}
We use the Taylor approximation that if $\delta \le \half$, then
$\log(1 - \delta) \ge -\delta - \delta^2$, so
\begin{equation*}
  \E[\exp(\lambda (Z^T B Z - \tr(B)))]
  \le \exp\left(\sum_{i = 1}^n \left(\lambda b_i + 2 \lambda^2 b_i^2
  \right) - \lambda \tr(B)\right)
  = \exp\left(2 \lambda^2 \tr(B^2)\right)
\end{equation*}
whenever $0 \le \lambda \le \frac{1}{4 \linf{b}}$.
If $\lambda \le 0$, an identical calculation holds when
$\lambda \ge -\frac{1}{4 \linf{b}}$.
This yields the first result of the lemma.

For the second, note that
$\opnorm{2 B} = \opnorms{A + A^T} \le \opnorm{A} + \opnorms{A^T}
= 2 \opnorm{A}$, while
\begin{equation*}
  \tr((A + A^T)^2)
  = \tr(A A) + \tr(AA^T) + \tr(A^TA) + \tr(A^T A^T)
  \le 4 \lfro{A}^2,
\end{equation*}
where we have used that $\<C, D\> = \tr(C^T D)$ is an inner product
on matrices and the Cauchy-Schwarz inequality.

\bibliography{bib}

\begin{thebibliography}{6}
\providecommand{\natexlab}[1]{#1}
\providecommand{\url}[1]{\texttt{#1}}
\expandafter\ifx\csname urlstyle\endcsname\relax
  \providecommand{\doi}[1]{doi: #1}\else
  \providecommand{\doi}{doi: \begingroup \urlstyle{rm}\Url}\fi

\bibitem[Borwein and Chan(2009)]{BorweinCh09}
J.~M. Borwein and O.-Y. Chan.
\newblock Uniform bounds for the incomplete complementary {G}amma function.
\newblock \emph{Mathematical Inequalities and Applications}, 12:\penalty0
  115--121, 2009.

\bibitem[Buja et~al.(1989)Buja, Hastie, and Tibshirani]{BujaHaTi89}
A.~Buja, T.~Hastie, and R.~Tibshirani.
\newblock Linear smoothers and additive models.
\newblock \emph{Annals of Statistics}, 17\penalty0 (2):\penalty0 453--555,
  1989.

\bibitem[Efron(2012)]{Efron12}
B.~Efron.
\newblock \emph{Large-Scale Inference: Empirical Bayes Methods for Estimation,
  Testing, and Prediction}.
\newblock Insitute of Mathematical Statistics Monographs. Cambridge University
  Press, 2012.

\bibitem[Tibshirani and Rosset(2019)]{TibshiraniRo19}
R.~J. Tibshirani and S.~Rosset.
\newblock Excess optimism: How biased is the apparent error of an estimator
  tuned by {SURE}?
\newblock \emph{Journal of the American Statistical Association}, 114\penalty0
  (526):\penalty0 697--712, 2019.

\bibitem[van~de Geer(2000)]{vandeGeer00}
S.~van~de Geer.
\newblock \emph{Empirical Processes in M-Estimation}.
\newblock Cambridge University Press, 2000.

\bibitem[Wainwright(2019)]{Wainwright19}
M.~J. Wainwright.
\newblock \emph{High-Dimensional Statistics: A Non-Asymptotic Viewpoint}.
\newblock Cambridge University Press, 2019.

\end{thebibliography}
\bibliographystyle{abbrvnat}

\end{document}